\documentclass[12pt, reqno]{amsart}

\usepackage{amsmath, amsthm, amscd, amsfonts, amssymb, graphicx, color}
\usepackage[bookmarksnumbered, colorlinks, plainpages]{hyperref}
\input{mathrsfs.sty}
\hypersetup{colorlinks=true,linkcolor=red, anchorcolor=green, citecolor=cyan, urlcolor=red, filecolor=magenta, pdftoolbar=true}

\newtheorem{theorem}{Theorem}[section]
\newtheorem{lemma}[theorem]{Lemma}

\newtheorem{corollary}[theorem]{Corollary}
\theoremstyle{definition}

\theoremstyle{remark}
\newtheorem{remark}[theorem]{Remark}
\numberwithin{equation}{section}

\begin{document}
\title{Gr\"uss inequality for some types of positive linear maps}
\author[J.S. Matharu and M.S. Moslehian]{Jagjit Singh Matharu and Mohammad Sal Moslehian}
\address{J. S. Matharu, Department of Mathematics, Bebe Nanaki University College,
Mithra, Kapurthla, Punjab, India.}
\email{matharujs@yahoo.com}
\address{M. S. Moslehian, Department of Pure Mathematics, Center of Excellence in
Analysis on Algebraic Structures (CEAAS), Ferdowsi University of
Mashhad, P.O. Box 1159, Mashhad 91775, Iran.}
\email{moslehian@um.ac.ir and moslehian@member.ams.org}


\begin{abstract}
Assuming a unitarily invariant norm $|||\cdot|||$ is given on a two-sided ideal of bounded linear operators acting on a separable Hilbert space, it induces some unitarily invariant norms $|||\cdot|||$ on matrix algebras $\mathcal{M}_n$ for all finite values of $n$ via $|||A|||=|||A\oplus 0|||$. We show that if $\mathscr{A}$ is a $C^*$-algebra of finite dimension
$k$ and $\Phi: \mathscr{A} \to \mathcal{M}_n$ is a unital completely
positive map, then
\begin{equation*}
|||\Phi(AB)-\Phi(A)\Phi(B)||| \leq \frac{1}{4}
|||I_{n}|||\,|||I_{kn}||| d_A d_B
\end{equation*}
for any $A,B \in
\mathscr{A}$, where $d_X$ denotes the diameter of the unitary orbit
$\{UXU^*: U \mbox{ is unitary}\}$ of $X$ and $I_{m}$ stands for the
identity of $\mathcal{M}_{m}$. Further we get an analogous
inequality for certain $n$-positive maps in the setting of full
matrix algebras by using some matrix tricks. We also give a Gr\"uss
operator inequality in the setting of $C^*$-algebras of arbitrary
dimension and apply it to some inequalities involving continuous
fields of operators.
\end{abstract}

\maketitle


\section{Introduction}

The Gr\" uss inequality \cite{GRU}, as a complement of Chebyshev's
inequality, states that if $f$ and $g$ are integrable real functions
on $[a, b]$ and there exist real constants $\varphi, \Phi, \gamma,
\Gamma$ such that $\varphi \leq f(x) \leq \Phi$ and $\gamma \leq
g(x) \leq \Gamma$ hold for all $x \in [a, b]$, then
\begin{equation}\label{MSM}
\frac{1}{b-a} \int_a^b f(x)g(x)dx - \frac{1}{(b-a)^2} \int_a^b
f(x)dx \int_a^b g(x)dx \leq \frac{1}{4}(\Phi- \varphi)(\Gamma -
\gamma)\,.
\end{equation}
The constant $1/4$ is the best possible and is achieved for
$f(x)=g(x)={\rm sgn}\big(x-(a+b)/2\big).$ It has been the
subject of much investigation in which the conditions on the
functions are varied to obtain different estimates. This inequality
has been investigated, applied and generalized by many
mathematicians in different areas of mathematics, such as inner
product spaces, quadrature formulae, finite Fourier transforms and
linear functionals; see \cite{DRA2} and references within. It has
been generalized for inner product modules over $H^*$-algebras and
$C^*$-algebras by Bani\'c, Ili\v{s}evi\'c and Varo\v{s}anec
\cite{B-I-V}. Renaud \cite{REN} gave matrix analogue of Gr\"uss
inequality by replacing integrable functions by normal matrices and
the integration by a trace function as follows: Let $A,B$ be square
matrices whose numerical ranges are lying in the circular discs of
radii $r$ and $s$, respectively. Then for a matrix $X$ of trace one,
\[
\left| \rm{tr}(XAB) - \rm{tr}(XA)\rm{tr} (XB) \right| \le krs\,,
\]
where $1\le k \le 4$. If $A$ and $B$ are normal, then $k=1$. Another
Gr\"uss type inequality involving the trace functional is given by
Bourin \cite{BOU}. Peri\'c and Raji\'c \cite{P-R} extended the
result of Renaud by showing that if $\Phi$ is a unital completely
bounded linear map from a unital $C^*$-algebra $\mathscr{A}$ to the
$C^*$-algebra of bounded operators on some Hilbert space
$\mathscr{H}$, then
$$\|\Phi(AB) - \Phi(A)\Phi(B)\| \leq \|\Phi\|_{cb}\,{\rm diam}(W^1(A))\,{\rm
diam}(W^1(B))$$ for every $A, B \in \mathscr{A}$, where
$W^1(\cdot)=\{\varphi(\cdot): \varphi \mbox{ is a state of }
\mathscr{A}\}$ denotes the generalized numerical range and
$\|\Phi\|_{cb}=\sup_n\|\Phi_n\|$. This result was extended by
Moslehian and Raji\'c \cite{M-R} for $n$-positive linear maps ($n
\geq 3$). In addition, Jocic, Krtinic and Moslehian \cite{JKM}
presented a Gr\"uss inequality for inner product type integral
transformers in norm ideals. Also, several operator Gr\"uss type
inequalities are given by Dragomir in \cite{DRA2} by utilizing the
continuous functional calculus and spectral resolution for
self-adjoint operators.

In this paper, we present a general Gr\"uss inequality for unital
completely positive maps and unitarily invariant norms.
Further we get a similar inequality for certain $n$-positive maps in the setting of full matrix algebras
by emplying some matrix tricks. We also give a Gr\"uss operator
inequality in the setting of $C^*$-algebras of arbitrary dimension
and apply it to inequalities involving continuous fields of
operators.

\section{Preliminaries}

Let ${\mathbb B}({\mathscr H})$ be the $C^*$-algebra of all bounded
linear operators on a complex (separable) Hilbert space $({\mathscr H}, \langle
\cdot , \cdot \rangle)$ and $I$ be its identity. Whenever $\dim
\mathscr{H}=n$, we identify $\mathbb{B}(\mathscr{H})$ with the the
full matrix algebra $\mathcal{M}_n$ of all $n \times n$ matrices
with entries in the complex field $\mathbb{C}$ and denote its
identity by $I_n$. We write $A\geq 0$ if $A$ is a positive operator
(positive semi-definite matrix) in the sense that $\langle A x,x
\rangle \geq 0$ for all $x\in \mathscr{H}$. Further, $A \geq B$ if
$A$ and $B$ are self adjoint operators and $A-B\geq 0$. Let $\mathbb{K}(\mathscr{H})$ denote the ideal of compact operators
on $\mathscr{H}$. For any operator $A\in \mathbb{K}(\mathscr{H})$,
let $s_{1}(A), s_{2}(A),\cdots$ be the eigenvalues of $|A|=
(A^*A)^{1\over{2}}$ in decreasing order and repeated according to
multiplicity. If $A \in \mathcal{M}_n$, we take
$s_k(A)=0$ for $k>n$.\\
Denote by $c_0$ the set of complex sequences converging to zero.
Consider the set $c_F \subseteq c_0$ of sequences with finite non-zero entries.
For $a \in c_0$, denote $\lfloor a\rfloor=(|a_n|)_{n \in \mathbb{N}} \in c_0$.
Following \cite[Section III.3]{GK}, a symmetric norming function (or symmetric gauge function for matrices \cite[p. 86]{bhatiaa}) is a map $g:c_F \to \mathbb{R}$ satisfying the properties
\begin{itemize}
\item[(i)] $g$ is a norm on $c_F$;
\item[(ii)] $g(a)=g(\lfloor a\rfloor)$ for every $a \in c_F$;
\item[(iii)] $g$ is invariant under permutations.
\end{itemize}
For $a=(a_i) \in c_0$, let us define $g(a)=\sup_{n \in \mathbb{N}} g(a_1,\ldots,a_n,0,\ldots) \in \mathbb{R} \cup \{+\infty\}$.\\
A unitarily invariant norm in $\mathbb{K}(\mathscr{H})$ is a map
$|||\cdot||| : \mathbb{K}(\mathscr{H}) \to [0,\infty]$
given by $|||A|||=g(s(A))$, $A \in \mathbb{K}(\mathscr{H})$, where $g$ is a symmetric norming function; see \cite[Chapter III]{GK}.
The set $\mathcal{C}_{|||\cdot|||}=\{A \in \mathbb{K}(\mathscr{H}) :|||A||| < \infty \}$
is a self-adjoint (two-sided) ideal of $\mathbb{B}(\mathscr{H})$. The Ky Fan norms as an example of unitarily invariant norms are defined as $\| A\|
_{(k)}=\sum_{j=1}^{k}s_{j}(A)$ for $k=1,2,\ldots$. The Ky Fan
dominance theorem \cite[Theorme IV.2.2]{bhatiaa} states that $\| A\|
_{(k)}\leq \| B\| _{(k)}\,\,(k=1,2,\ldots )$ if and only if
$|||A||| \leq |||B|||$ for all unitarily invariant norms
$|||\cdot|||$. It is known that the Schatten $p$-norms
$\|A\|_p=\left(\sum_{j=1}^\infty s_j^p(A)\right)^{1/p}$ are also unitarily
invariant norms for $p \geq 1$; cf. \cite[Section IV.2]{bhatiaa}. Another example of a unitarily invariant norm is the usual operator
norm $\|\cdot\|$. The notation $A\oplus B$ is used for the block matrix
$\begin{pmatrix} A & 0 \\ 0 & B \end{pmatrix}$. It should be noted
that $\|A\oplus B\|=\max\{\|A\|,\|B\|\}$. \\
Throughout the paper we assume that a unitarily invariant norm $|||\cdot|||$ is given on a two-sided ideal of bounded linear operators acting on a separable Hilbert space and then the norms $|||\cdot|||$ on matrix algebras $\mathcal{M}_n$ for all finite values of $n$ are induced by it via
\begin{align}\label{notation}
|||A|||=|||A\oplus 0|||\,.
\end{align}
Thus we indeed deal with a system of unitarily invariant norms $\{|||\cdot|||_s\}$ on algebras $\mathcal{M}_s,\,\, s \leq N$ or on all algebras $\mathcal{M}_s,\,\, s \geq 1$ satisfying the relation $|||A|||_s=|||A\oplus 0_{(t-s)(t-s)}|||_t,\,\, A \in \mathcal{M}_s, t>s $ between norms of matrices of different sizes.

The unitary orbit of an operator $A$ is defined as the set of all operators
of the form $UAU^*$, where $U$ is a unitary. The diameter of the
unitary orbit is
\begin{equation*}
d_A=\sup \{\|AU-UA\|: U~ \mbox{is~unitary}\}= \sup_{\|X\|=1}
\|AX-XA\|= 2 \Delta(A, \mathbb{C} I)\,,
\end{equation*}
where $ \Delta(A, \mathbb{C}I)=\inf _{\lambda \in \mathbb{C}} \|A-
\lambda I\| $ is the $\|\cdot\|$-distance of $A$ from the scalar
operators; see \cite{STA}.

A linear map $\Phi: \mathscr{A} \to \mathscr{B}$ between
$C^*$-algebras is called positive if $\Phi(A)\geq 0 $ whenever $A\ge
0$ and is called unital if $\Phi$ preserves the identity in the case
that both $C^*$-algebras $\mathscr{A}, \mathscr{B}$ are unital.
Without any ambiguity we denote the identity of a $C^*$-algebra
$\mathscr{A}$ by $I$ as well. It follows from the linearity of a
positive map that $\Phi(A^*)=\Phi(A^*)$ for any $A \in \mathscr{A}$.
Let $\mathcal{M}_n (\mathscr{A})$ denotes the $n \times n$ block
matrix with entries from $\mathscr{A}$. Each linear map $\Phi:
\mathscr{A} \to \mathscr{B}$ induces a linear map $\Phi_n$ from $
\mathcal{M}_n(\mathscr{A})$ to $\mathcal{M}_n(\mathscr{B})$ defined
by $\Phi_n([A_{ij}]_{n\times n})=[\Phi(A_{ij})]_{n \times n}$. We
say that $\Phi$ is $n$-positive if the map $\Phi_{n}$ is positive
and $\Phi$ is completely positive if the maps $\Phi_{n}$ are
positive for all $n=1,2,\ldots $. It is a known that due to
Stinespring that the restriction of any positive linear map to a
unital commutative $C^*$-algebra is completely positive,
\cite[Theorem 4]{STI}.

\section{Gr\"uss inequality for the finite dimensional case}
To achieve our main result we need the following well-known lemmas.
The first lemma is an immediate consequence of the min-max principle
and the Ky Fan dominance theorem.

\begin{lemma}
\label{lem3.1} \cite[p. 75]{bhatiaa} Let $A, X, B\in \mathcal{M}_n$.
Then
\begin{itemize}
\item[(i)] $s_{j}(AXB)\leq \|A\|\,s_{j}(X)\,\|B\| \quad (j=1,2,\ldots,
n)$.\\
\item[(ii)] $|||AXB||| \leq \|A\|\,\,|||X|||\,\,\|B\|$.
\end{itemize}
\end{lemma}
The next lemma gives an estimate of $|||K|||$ when $K$ is a
contraction, i.e. a matrix of operator norm less than or equal one.
\begin{lemma}
\label{q} Let $|||.|||$ be a unitarily invariant norm on
$\mathcal{M}_n$. If $K$ is a contraction, then
$$|||K||| \leq |||I_n|||\,.$$
\end{lemma}
\begin{proof}
It follows from Lemma \ref{lem3.1} (ii) that
\begin{align*}|||K||| = |||~|K|~||| &= |||~|K|^{\frac{1}{2}}~ I_n ~ |K|^{\frac{1}{2}}~|||\\
& \leq \|~|K|^{\frac{1}{2}}~\| ~|||I_n|||~ \|~|K|^{\frac{1}{2}}~\|
=\|K\|~|||I_n||| \leq |||I_n|||.
\end{align*}
\end{proof}
The two next lemmas deal with the positivity of block matrices.
\begin{lemma}\cite[Corollary I.3.3]{bhatiaa}
\label{AAAA} Let $A\in \mathcal{M}_n$. Then $A$ is positive if and
only if the block matrix $\left(
\begin{array}{cc}
A & A \\
A & A%
\end{array}%
\right) $ is positive.
\end{lemma}
\begin{lemma}\cite[Theorme IX.5.9]{bhatiaa} \label{cont}
Let $A,B \in \mathcal{M}_n$ be positive. Then the block matrix
$\begin{pmatrix} A & X \\ X^* & B \end{pmatrix}$ is positive if and
only if $ X=A^{1/2}KB^{1/2}$ for some contraction $K$.
\end{lemma}

The next lemma is known as Horn's Theorem.
\begin{lemma}\cite[Corollary 10.3]{ZHA} \label{horn}
Let $A,B \in \mathcal{M}_n$. Then
$$\prod_{i=1}^ks_j(AB) \leq \prod_{i=1}^k\big(s_j(A)s_j(B)\big)\,\,\,\,\,\,(k=1, 2, \ldots, n)$$
\end{lemma}
The celebrated Stinespring dilation theorem \cite[Theorem 1]{STI}
states that for any unital completely positive map $\Phi:
\mathscr{A}\to \mathbb{B}(\mathscr{H})$ between $C^*$-algebras there
exist a Hilbert space $\mathscr{K}$, an isometry $V: \mathscr{H}\to
\mathscr{K}$ and a unital $*$-homomorphism $\pi: \mathscr{A}\to
\mathbb{B}(\mathscr{K})$ such that $\Phi(T)=V^*\pi (T)V$ for all $T
\in \mathscr{A}$; see also \cite{ACM} and reference therein. We
assume that $\mathscr{K}$ is the closure of
$\pi(\mathscr{A})V\mathscr{H}$ and then we get the minimal
Stinespring representation which is unique up to a unitary
equivalence. Moreover, if $\dim(\mathscr{A})=k$ and $\dim(\mathscr{H})=n$, then
$\dim(\mathscr{K}) \le kn$. The equality occurs if $\mathscr{A}= \mathcal{M}_m$ for some $m$, see \cite[Theorem 3.1.2]{bhatiab}. Therefore we deduce that $$|||I_{\dim(\mathscr{K})}||| \leq |||I_{kn}|||\,,$$
since, by the Fan dominance theorem and Weyl's monotonicity theorem \cite[p. 63]{bhatiaa}, a sufficient condition to have $|||A||| \leq |||B|||$ is that $A\leq B$. Thus it is meaningful to deal with singular values of elements of
$\mathbb{B}(\mathscr{K})$.

We are ready to establish our first main result. The first part is a
Kantorovich additive type inequality and the second is a Gr\"uss
type one.

\begin{theorem}\label{main1}
Let $\mathscr{A}$ be a finite dimensional $C^*$-algebra of dimension
$k$ and $\Phi: \mathscr{A} \to \mathcal{M}_n$ be a unital completely
positive map. Then
\begin{eqnarray*}
\hspace{-1.9 in}{\rm (i)~}
|||\Phi(A^*A)-\Phi(A^*)\Phi(A)|||^{\frac{1}{2}} \leq
\frac{1}{2}\sqrt{|||I_{kn}|||}d_A\,
\end{eqnarray*}
for all $A \in
\mathscr{A}$.
\begin{eqnarray*}
\hspace{-1.75in}{\rm (ii)~} |||\Phi(AB)-\Phi(A)\Phi(B)||| \leq
 \frac{1}{4} |||I_{n}|||\,|||I_{kn}|||d_A d_B
\end{eqnarray*}
for all $A,B \in \mathscr{A}$.
\end{theorem}
\begin{proof} (i) By using the Stinespring dilation theorem the positivity of
\begin{equation}
\left(
\begin{array}{cc}
\Phi(A^*A)-\Phi(A^*)\Phi(A) & \Phi(A^*B)-\Phi(A^{*
})\Phi(B) \\
\Phi(B^*A)-\Phi(B^*)\Phi(A) & \Phi(B^*B)-\Phi(B^{*
})\Phi(B)%
\end{array}%
\right) \label{eq1}
\end{equation}
will follow once we prove the positivity of
{\small\begin{align} \label{eq2}
\left(
\begin{array}{cc}
V^*\pi(A^*A)V-V^*\pi(A^*)VV^*\pi(A)V & V^*\pi(A^{*
}B)V-V^*\pi(A^*)VV^*\pi(B)V \\
V^*\pi(B^*A)V-V^*\pi(B^*)VV^*\pi(A)V & V^*\pi(B^{*
}B)V-V^*\pi(B^*)VV^*\pi(B)V%
\end{array}%
\right).
\end{align} }
As $V$ is an isometry, we have $VV^*\leq
I_{\dim(\mathscr{K})}$. It follows from Lemma \ref{AAAA} that
$$\begin{pmatrix}
VV^* & VV^*\\VV^* & VV^*
\end{pmatrix} \leq \begin{pmatrix}
I_{\dim(\mathscr{K})}& I_{\dim(\mathscr{K})}\\
I_{\dim(\mathscr{K})}& I_{\dim(\mathscr{K})}
\end{pmatrix}\,.$$
Hence \begin{align*}\begin{pmatrix}
\pi(A)^*& 0\\
0 & \pi(B)^*
\end{pmatrix}&\begin{pmatrix}
VV^* & VV^*\\VV^* & VV^*
\end{pmatrix}\begin{pmatrix}
\pi(A)& 0\\
0 & \pi(B)
\end{pmatrix}\\
&\leq \begin{pmatrix}
\pi(A)^*& 0\\
0 & \pi(B)^*
\end{pmatrix}\begin{pmatrix}
I_{\dim(\mathscr{K})}& I_{\dim(\mathscr{K})}\\
I_{\dim(\mathscr{K})}& I_{\dim(\mathscr{K})}
\end{pmatrix}\begin{pmatrix}
\pi(A)& 0\\
0 & \pi(B)
\end{pmatrix}\,,
\end{align*}
whence
\begin{equation}
\begin{pmatrix}
\pi(A)^*VV^*\pi(A) & \pi(A)^*VV^*\pi(B)\\
\pi(B)^*VV^*\pi(A) & \pi(B)^*VV^*\pi(B)
\end{pmatrix} \leq \begin{pmatrix}
\pi(A^*A) & \pi(A^*B)\\
\pi(B^*A) & \pi(B^*B)
\end{pmatrix}.
\label{eq233}
\end{equation}
The positivity of (\ref{eq2}) follows by pre-multiplying
(\ref{eq233}) by $\begin{pmatrix}
V^* & 0\\
0 & V^*
\end{pmatrix}$ and post-multiplying by $\begin{pmatrix}
V & 0\\
0 & V
\end{pmatrix}$.
It is notable that the positivity of \eqref{eq1} implies the
positivity of its $(1,1)$ entry:
\begin{equation*}
\Phi(A^*A)-\Phi(A^*)\Phi(A)\geq 0\qquad \mbox{(the so-called Kadison
inequality)}.
\end{equation*}%
Utilizing the Stinespring theorem we have
\begin{align*}
 \Phi(A^*A)-&\Phi(A^*)\Phi(A)\\
	&= V^*\pi(A^*A)V-V^* \pi(A^*)VV^* \pi(A)V\\
 &=V^*\pi\big((A-\lambda I)^*(A-\lambda I)\big)V-V^* \pi(A-\lambda I)^*VV^* \pi(A-\lambda I) V\\
&=V^* \pi(A-\lambda I)^*(I_{\dim(\mathscr{K})}-VV^*) \pi(A-\lambda I)V
\end{align*}
for every $\lambda \in \mathbb{C}$.\\
Note that $I_{\dim(\mathscr{K})}-VV^*$ is a
projection and $\pi$ is a $*$ -homomorphism, hence
\begin{align}\label{log}
&\prod_{j=1}^{k}s_j\Big(\Phi(A^*A)-\Phi(A^*)\Phi(A)\Big)\nonumber\\
&= \prod_{j=1}^{k}s_j\Big( V^*\pi(A-\lambda I)^*(I_{\dim(\mathscr{K})}-VV^*)\pi(A-\lambda I)V \Big)\nonumber\\
&\leq \prod_{j=1}^{k}\Big[s_j\Big( V^*\pi(A-\lambda
I)^*(I_{\dim(\mathscr{K})}-VV^*)\Big)\,s_j\Big((I_{\dim(\mathscr{K})}-VV^*)\pi(A-\lambda
I)V \Big)\Big]\nonumber\\
&\qquad\qquad\qquad\qquad\qquad\qquad\qquad\qquad \qquad \qquad \qquad(\mbox{by Lemma \ref{horn}})\nonumber\\
&\leq \prod_{j=1}^{k}\Big[s_j\big(\pi(A-\lambda
I)^*\big)\,s_j\big(\pi(A-\lambda
I) \big)\Big]\qquad\qquad(\mbox{by Lemma \ref{lem3.1} (i)})\nonumber\\
&= \prod_{j=1}^{k}s_j(\pi(|A-\lambda I|^2))\nonumber\\
&\qquad\quad(\mbox{since eigenvalues of matrices $XY$
and $YX$ are the same})\,.
\end{align}
for all $k=1,2,\ldots, n$ and $\lambda \in \mathbb{C}$. Since the
weak log-majorization inequality implies the weak majorization
inequality (cf. \cite[Theorem 10.15]{ZHA}), we get from \eqref{log}
that
\begin{align*}
\sum_{j=1}^{k}s_j\left(\Phi(A^*A)-\Phi(A^*)\Phi(A)\right)\leq
\sum_{j=1}^{k}s_j(\pi(|A-\lambda I|^2))\qquad (k=1,2,\ldots, n)\,.
\end{align*}
Thus, by using Lemma \ref{lem3.1} (ii), we reach
\begin{align*}
|||\Phi(A^*A)-\Phi(A^*)\Phi(A)||| &\leq |||\pi(|A-\lambda I|^2)|||\\
& = |||\pi(|A-\lambda I|)\,I_{\dim(\mathscr{K})}\, \pi(|A-\lambda I|)|||\\
& \leq \|\pi(|A-\lambda I|)\|\,\,|||I_{\dim(\mathscr{K})}|||\,\,\|\pi(|A-\lambda
I|)\| \\
& \leq \|A-\lambda I\|^2\,|||I_{kn}|||\quad
(\mbox{since $\pi$ is norm decreasing})\,.
\end{align*}
Therefore $$|||\Phi(A^*A)-\Phi(A^*)\Phi(A)|||^{\frac{1}{2}} \leq
\sqrt{|||I_{kn}|||} \inf _{\lambda \in \mathbb{C}}\| A-\lambda
I\|=\sqrt{|||I_{kn}|||}d_A\,.$$

(ii) Since (\ref{eq1}) is positive, by Lemma \ref{cont}, there
exists a contraction $K\in \mathcal{M}_n$ such
that%
{\small \begin{align*}
\Phi(A^*B)-\Phi(A^*)\Phi(B)=\left( \Phi(A^*A)-\Phi(A^{* })\Phi
(A)\right) ^{\frac{1}{2}}K\left( \Phi(B^*B)-\Phi(B^*)\Phi(B)\right)
^{\frac{1}{2}}.
\end{align*}}
It follows that
\begin{align*}
|||\Phi(A^*B)-&\Phi(A^*)\Phi(B)|||\\
 &=|||\left( \Phi(A^*A)-\Phi(A^{*
})\Phi (A)\right) ^{\frac{1}{2}}K\left(
\Phi(B^*B)-\Phi(B^*)\Phi(B)\right)
^{\frac{1}{2}}|||\\
& \leq \|\Phi(A^*A)-\Phi(A^{* })\Phi (A)\|
^{\frac{1}{2}}|||K|||\,\,\|
\Phi(B^*B)-\Phi(B^*)\Phi(B)\|^{\frac{1}{2}}\\
&\qquad\qquad\qquad\qquad\qquad\qquad\qquad\qquad\qquad(\mbox{by Lemma \ref{lem3.1} (ii)})\\
& \leq \|\Phi(A^*A)-\Phi(A^{* })\Phi (A)\|
^{\frac{1}{2}}|||I_{n}|||\,\,\|
\Phi(B^*B)-\Phi(B^*)\Phi(B)\|^{\frac{1}{2}}\\
&\qquad\qquad\qquad\qquad\qquad\qquad\qquad\qquad\qquad(\mbox{by Lemma \ref{q}})\\
&\leq |||I_{n}|||\,|||I_{kn}||| \inf _{\lambda \in \mathbb{C}}\|
A-\lambda I\|\inf _{\mu \in \mathbb{C}}\| B-\mu I\|\qquad(\mbox{by
part (i)})\\
&= \frac{1}{4} |||I_{n}|||\,|||I_{kn}||| d_A d_B\,.
\end{align*}
The result follows by replacing $A^*$ by $A$ in the last inequality.
\end{proof}
As a consequence we get the following Gr\"uss inequalities for some known unitarily invariant norms.

\begin{corollary}
If $\Phi: \mathcal{M}_m \to \mathcal{M}_n$ is a unital completely
positive map, then
\[
\| \Phi(AB)-\Phi(A)\Phi(B)\| \leq \frac{1}{4} \max_{\|X\|=1,
\|Y\|=1} \|AX-XA\|~ \|BY-YB\|\leq \frac{1}{4} d_A d_B\,.
\]
and
\[
\|\Phi(AB)-\Phi(A)\Phi(B)\|_p \leq \frac {(mn)^{2/p}}{4}d_A d_B\, \qquad (p \geq 1)
\]
for all $A,B \in \mathcal{M}_m$.
\end{corollary}
\begin{proof}
First observe that $\dim(\mathcal{M}_m)=m^2$. Second note that the
operator norm $\|\cdot\|$ and the Schatten $p$-norm $\|\cdot\|_p$
whenever $p \geq 1$ are unitarily invariant norms as well as
$\|I_k\|_p=k^{1/p}$ for every positive integer $k \geq 1$. It is now
sufficient to use Theorem \ref{main1}.
\end{proof}

If $A$ is self-adjoint, with $mI\leq A \leq MI$ for some real
numbers $m,M$, then $d_A = M-m$. As a consequence of Theorem
\ref{main1} we have the following result.

\begin{corollary}\label{cor2}
Let $\Phi: \mathcal{M}_m \to \mathcal{M}_n$ be a unital completely
positive map and $A,B \in \mathcal{M}_n$ be Hermitian matrices with
$mI_m\leq A \le MI_m,~m'I_m\leq B \leq M'I_m$ for some constants $m,
m', M, M'$. Then
\begin{equation*}
|||\Phi(AB)-\Phi(A)\Phi(B)||| \leq \frac{1}{4}
(M-m)(M'-m')|||I_n|||\,|||I_{m^2n}|||\,.
\end{equation*}
\end{corollary}
In \cite{M-R} we estimate the operator norm of
$\Phi(AB)-\Phi(A)\Phi(B)$ for an $n$-positive linear map $\Phi$. Now
we estimate any its unitarily invariant norm. We need the next two
lemmas. The first is an equivalent version of Lemma \ref{cont}.

\begin{lemma}\cite{AND} \label{lemrb}
Let $C\in \mathbb{B}(\mathscr{H}_1), D \in
\mathbb{B}(\mathscr{H}_2)$ be positive and $D$ be invertible. Then
the block matrix $\begin{pmatrix} C & X
\\X^* & D
\end{pmatrix}$ is positive if and only if $C \geq XD^{-1} X^*$.
\end{lemma}

\begin{lemma} \label{ppp} For any matrix $X \in \mathcal{M}_n$,
\begin{eqnarray*}
\left\|\begin{bmatrix}0& X\\ X^*& 0\end{bmatrix}\right\|= \|X\|\,.
\end{eqnarray*}
\end{lemma}
\begin{proof}
\begin{align*}
\left\|\begin{bmatrix} 0& X\\ X^*&
0\end{bmatrix}\right\|&=\left\|\begin{bmatrix} I_n& 0\\0&
I_n\end{bmatrix}\begin{bmatrix} 0& X\\ X^*&
0\end{bmatrix}\begin{bmatrix} 0 &I_n\\I_n &
0\end{bmatrix}\right\| \\
&=\left\|\begin{bmatrix} X& 0\\0&
X^*\end{bmatrix}\right\|=\max\{\|X\|,\|X^*\|\}=\|X\|\,.
\end{align*}
\end{proof}

Now we ready to extend the main theorem \cite{M-R} in some
directions by using some matrix tricks.

\begin{theorem}\label{the2}
Let $ 12 \le \eta$ be a positive integer and let $\Phi: \mathcal{M}_m
\to \mathcal{M}_n$ be a unital $\eta$-positive linear map. Then
\begin{align}\label{nice}
 |||\Phi(AB)-\Phi(A)\Phi(B)||| \leq \frac{1}{4}|||I_{n}|||~|||I_{m^2n}|||\,d_A\,d_B
\end{align}
for all $A, B \in \mathcal{M}_m$.
\end{theorem}
\begin{proof}

First assume that $A, B$ are Hermitian matrices. Employing the
$3$-positivity of linear map $\Phi$ to the $3\times 3$ block matrix
$$
\begin{bmatrix}
A^*A&A^*B&A^*\\
B^*A&B^*B&B^*\\
A&B&I_m
\end{bmatrix}=\begin{bmatrix}
A& B& I_m
\end{bmatrix}^*
\begin{bmatrix}
A& B& I_m
\end{bmatrix} \geq 0$$
and applying Lemma \ref{lemrb} with $X^*=\Big(\Phi(A)~~\Phi(B) \Big)$, $C=\left(
\begin{array}{cc}
\Phi(A^*A) & \Phi(A^*B) \\
\Phi(B^*A) & \Phi(B^*B)%
\end{array}%
\right)$ and $D=I_m$, we obtain
\begin{equation}
\left(
\begin{array}{cc}
\Phi(A^*A)-\Phi(A)^*\Phi(A) & \Phi(A^*B)-\Phi(A)^{*
}\Phi(B) \\
\Phi(B^*A)-\Phi(B)^*\Phi(A) & \Phi(B^*B)-\Phi(B)^{*
}\Phi(B)%
\end{array}%
\right) \geq 0. \label{eq222}
\end{equation}
As $A$ is Hermitian, the unital $C^*$-algebra $\mathscr{C}^*(A,I_m)$
generated by $A$ and the identity $I_m$ is commutative. Hence the
restriction of $\Phi$ to $\mathscr{C}^*(A,I_m)$ is a unital
completely positive map. Thus Theorem \ref{main1} (i) gives us the
inequality
$$|||\Phi(A^*A)-\Phi(A^*)\Phi(A)|||^{\frac{1}{2}} \leq
~\sqrt{|||I_{m^2n}|||}~\|A\|\,.$$ A similar formula is valid for
$B$ instead of $A$. Now the same reasoning as in the proof of
Theorem \ref{main1} (ii) along with \eqref{eq222} shows that the
inequality
\begin{equation}\label{zzz}
|||\Phi(AB)-\Phi(A)\Phi(B)||| \leq
|||I_{n}|||~|||I_{m^2n}|||~\|A\|~\|B\|
\end{equation}
holds for any Hermitian matrices $A, B$ and any $3$-positive map
$\Phi$.

Second let $A$ and $B$ be arbitrary and Hermitian matrices,
respectively. Applying inequality \eqref{zzz} to $3$-positive map
$\Phi_2: \mathcal{M}_2(\mathcal{M}_m) \to
\mathcal{M}_2(\mathcal{M}_n)$ and Hermitian matrices $\begin{bmatrix} 0& A\\
A^*& 0\end{bmatrix}$ and $\begin{bmatrix} 0&0\\
0& B\end{bmatrix}$ and using Lemma \ref{ppp} we get
\begin{align*}
&\left|\left|\left|\Phi_2\left(\begin{bmatrix} 0& A\\
A^*& 0\end{bmatrix}\begin{bmatrix} 0&0\\
0& B\end{bmatrix}\right)-\Phi_2\left(\begin{bmatrix} 0& A\\ A^*&
0\end{bmatrix}\right)\Phi_2\left(\begin{bmatrix} 0&0\\
0& B\end{bmatrix}\right)\right|\right|\right| \\
& \qquad \qquad\qquad\qquad\qquad\qquad\qquad\qquad\qquad\qquad\qquad\leq
|||I_{n}|||~|||I_{m^2n}|||~\|A\|\,\|B\|
\end{align*}
Since
$$\Phi_2\left(\begin{bmatrix} 0& A\\
A^*& 0\end{bmatrix}\begin{bmatrix} 0&0\\
0& B\end{bmatrix}\right)=\begin{bmatrix} 0& \Phi(AB)\\
0&0\end{bmatrix}$$ and $$\Phi_2\left(\begin{bmatrix} 0& A\\
A^*&
0\end{bmatrix}\right)\Phi_2\left(\begin{bmatrix} 0&0\\
0& B\end{bmatrix}\right)=\begin{bmatrix} 0& \Phi(A)\Phi(B)\\
0&0\end{bmatrix}$$ we have
$$\left|\left|\left|\begin{bmatrix} 0& \Phi(AB)-\Phi(A)\Phi(B)\\
0&0\end{bmatrix}\right|\right|\right|\leq
|||I_{n}|||~|||I_{m^2n}|||~\|A\|\,\|B\|\,.$$ Hence
\begin{align*}
&\hspace{-1cm}\left|\left|\left|\begin{bmatrix} \Phi(AB)-\Phi(A)\Phi(B)&0 \\
0&0\end{bmatrix}\right|\right|\right|\\
&=\left|\left|\left|\begin{bmatrix}I_n& 0\\
0&I_n\end{bmatrix}\begin{bmatrix} 0& \Phi(AB)-\Phi(A)\Phi(B)\\
0&0\end{bmatrix}\begin{bmatrix}0& I_n\\
I_n&0\end{bmatrix}\right|\right|\right|\\
&=\left|\left|\left|\begin{bmatrix} 0& \Phi(AB)-\Phi(A)\Phi(B)\\
0&0\end{bmatrix}\right|\right|\right|\\
&\leq|||I_{n}|||~|||I_{m^2n}|||~\|A\|~\|B\|
\end{align*}
for arbitrary matrix $A$, Hermitian matrix $B$ and $6$-positive map
$\Phi$.

Third, by repeating the same argument as above to the latter
inequality for arbitrary matrix $B$ we conclude that
\begin{align*}
&\hspace{-1cm}\left|\left|\left|\begin{bmatrix} \Phi(AB)-\Phi(A)\Phi(B)&0 \\
0&0\end{bmatrix}\right|\right|\right| \leq|||I_{n}|||~|||I_{m^2n}|||~\|A\|~\|B\|
\end{align*}
or, in our notation \eqref{notation},
$$|||\Phi(AB)-\Phi(A)\Phi(B) |||
\leq |||I_{n}|||\,|||I_{m^2n}|||\|A\|\,\|B\|$$ for any arbitrary
matrices $A, B$ and $12$-positive map $\Phi$. It follows from the
latter inequality that
\begin{align*} |||\Phi(AB)-&\Phi(A)\Phi(B)|||\\
&=|||\Phi\big((A-\lambda I_m)(B-\mu I_m)\big)-\Phi(A-\lambda
I_m)\Phi(B
-\mu I_m)|||\\
&\leq |||I_{n}|||~|||I_{m^2n}|||~\|A-\lambda I_m\|~\|B-\lambda
I_m\|\,. \end{align*} for all $\lambda,\mu\in \mathbb C.$ Thus
\begin{align*} ||| \Phi(AB)-\Phi(A)\Phi(B)|||&\leq |||I_{n}|||~|||I_{m^2n}|||\inf_{\lambda\in \mathbb
C}\|A-\lambda I_m\|\,\inf_{\mu\in\mathbb C}\|B-\mu I_m\|\\&=
\frac{1}{4}|||I_{n}|||~|||I_{m^2n}|||d_A\,d_B\,.
\end{align*}
\end{proof}

\begin{remark}
It is remarked that Theorem \ref{the2} is not true if $\Phi$ is
supposed to be unital $2$-positive linear map. To see this choose
map $\Phi:\mathcal{M}_3 \to \mathcal{M}_3$ defined as
$\Phi(A)=2\rm{tr}(A)I_3 -A$.
Then $\Phi$ is 2-positive but not 3-positive (see \cite{CHO}). Taking $A=\begin{pmatrix} 1 & 0& 0\\ 0 & 0 & 1\\
 0 & 1 & 0 \end{pmatrix}$ and $B =\begin{pmatrix}
 0 & 0 & 1\\
 0 & 0 & 0\\
 1 & 0 & 1 \end{pmatrix}$ one can easily observe that $2 \times 2$ block matrix in
(\ref{eq222}) is not positive and \eqref{nice} does not hold for the
operator norm. The case when $2<\eta<12$ remains unsolved.
\end{remark}

\section{Gr\"uss inequality for the case of arbitrary dimension}

A variant of the following lemma can be found in \cite[Lemma
4.1]{MNS}. We, however, prove it for the sake of
completeness. Recall that the ball of diameter $[x,y]$ in a normed
space $E$ is the set of all elements $z \in E$ such that
$\|z-(x+y)/2\| \leq \| (x-y)/2 \|$.

\begin{lemma}\label{lem1}
Let $\mathscr{A}$ be a unital $C^*$-algebra, $\Phi:\mathscr{A} \to
\mathbb{B}(\mathscr{H})$ be a unital completely positive map, $A \in
\mathscr{A}$ belongs to the ball of diameter $[mI, MI]$ for some
complex numbers $m, M$. Then
\begin{eqnarray*}
\Phi(|A|^2) - \left|\Phi(A)\right|^2 \leq \frac{1}{4}|M-m|^2 I.
\end{eqnarray*}
\end{lemma}
\begin{proof}
For any complex number $c\in {\mathbb C}$, we have
\begin{equation} \label{eq:Gr-1}
\Phi(|A|^2)-|\Phi(A)|^2 = \Phi(|A-c|^2)-|\Phi(A-c)|^2.
\end{equation}
The assumption of lemma implies that
\[
\left|A-\frac{M+m}{2}I\right|^2 \leq \frac{1}{4}|M-m|^2I
\]
whence
\begin{equation} \label{eq:Gr-2}
\Phi\left(\left|A-\frac{M+m}{2}I\right|^2\right) \leq
\frac{1}{4}|M-m|^2I\,.
\end{equation}
It follows from \eqref{eq:Gr-1} and \eqref{eq:Gr-2} that
\begin{eqnarray*}
\Phi(|A|^2)-|\Phi(A)|^2 \leq
\Phi\left(\left|A-\frac{M+m}{2}I\right|^2\right) \leq
\frac{1}{4}|M-m|^2I.
\end{eqnarray*}
\end{proof}
\begin{remark}
The geometric property that $A \in \mathscr{A}$ belongs to the ball
of diameter $[mI, MI]$ in Lemma \ref{lem1} is equivalent to the fact
that
\begin{eqnarray*}
{\rm Re}\left((MI-A)^*(A-mI)\right)\geq 0
\end{eqnarray*}
since
\[
\frac{1}{4}|M-m|^2I-\left|A-\frac{M+m}{2}\right|^2=\mbox{Re}\left((MI-A)^*(A-mI)\right)\,.
\]
\end{remark}
We are ready to state our second main result. It should be notified that it is an operator inequality of Gr\"uss type while inequality (ii) in Theorem \ref{main1} provides a Gr\"uss norm inequality.

\begin{theorem}\label{main2}
Let $\mathscr{A}$ be a unital $C^*$-algebra and $\Phi:\mathscr{A}
\to \mathbb{B}(\mathscr{H})$ be a unital completely positive map. If
elements $A$ and $B$ of $\mathscr{A}$ belong to the balls of
diameter $[m_1I, M_1I]$ and $[m_2I, M_2I]$ for some complex numbers
$m_1, M_1, m_2,M_2$, respectively, then
\begin{eqnarray*}
\left| \Phi(AB)-\Phi(A)\Phi(B)\right| \leq \frac{1}{4} |M_1-m_1| \,
|M_2-m_2|I.
\end{eqnarray*}
\end{theorem}
\begin{proof}
Let us use the notation in Theorem \ref{main1}. Using the positivity
of block matrix (\ref{eq2}) and Lemma \ref{lemrb}, we get
\begin{align*}
\frac{1}{4} |M_1&-m_1|^2I \\
&\geq V^*\pi(|A|^2)V - \big |V^*\pi(A)V \big|^2 \nonumber\\
 & \geq \big( V^*\pi(A^*B)V-V^*\pi(A)^*VV^*\pi(B)V \big) \left(V^*\pi(|B|^2)V- \big |V^*\pi(B)V \big|^2\right)^{-1}\nonumber\\
&\qquad\qquad\qquad\qquad\qquad\qquad \times \big( V^*\pi(A^*B)V-V^*\pi(A)^*VV^*\pi(B)V\big)^*\nonumber\\
&\geq \frac{4}{|M_2-m_2|^2} \left|
V^*\pi(A)^*BV-V^*\pi(A)^*VV^*\pi(B)V \right|^2,
\end{align*}
where the first and the third inequalities follow from Lemma
(\ref{lem1}) by taking the positive linear map $\Phi(X)=V^*\pi(X)V$,
where $V$ is the isometry in the Stinespring theorem. Hence
\begin{eqnarray}\label{nep}
\big| V^*\pi(A^*B)V-V^*\pi(A^*)VV^*\pi(B)V\big| \leq \frac{1}{4} \,
|M_1-m_1|\, |M_2-m_2|I.
\end{eqnarray}
Now the use of the Stinespring theorem yields
\begin{eqnarray*}
\left| \Phi(A^*B)-\Phi(A)^*\Phi(B)\right| \leq \frac{1}{4} |M_1-m_1|
\, |M_2-m_2|I.
\end{eqnarray*}
Replacing $A^*$ by $A$ in the latter inequality gives us the desired
inequality.
\end{proof}

Let $X \otimes Y$ and $X \circ Y$ denote the tensor product and the
Hadamard product of matrices $X$ and $Y$, respectively. Taking
$A=A_1 \otimes A_2$, $B=B_1 \otimes B_2$, $*$-homomorphism
$\pi(X)=X$ and isometry $V$ as a selective operator with property
$V^*(X \otimes Y)V=X \circ Y$ in (\ref{nep}) we get the following
corollary as a Hadamard product version of Gr\"uss inequality.

\begin{corollary}
 Let $A_1,A_2,B_1,B_2 \in \mathcal{M}_n$ such that matrices
$A_1 \otimes A_2$ and $B_1 \otimes B_2$ belong to the balls of
diameter $[m_1I_{n^2}, M_1I_{n^2}]$ and $[m_2I_{n^2}, M_2I_{n^2}]$
for some complex numbers $m_1, M_1, m_2,M_2$, respectively. Then
\begin{eqnarray*}
\left| (A_1B_1)\circ (A_2B_2)-(A_1\circ A_2)(B_1\circ B_2)\right|
\leq \frac{1}{4} |M_1-m_1| \, |M_2-m_2|I.
\end{eqnarray*}
\end{corollary}

Let ${\mathscr A}$ be a unital $C^*$-algebra and let $T$ be a
locally compact Hausdorff space. Let $C(T,\mathscr{A})$ be the set
of bounded continuous functions on $T$ with values in $\mathscr{A}$
as a normed involutive algebra by applying the point-wise operations
and settings. By a field $(A_t)_{t\in T}$ of operators in ${\mathscr
A}$ we mean a function of $T$ into $\mathscr{A}$. It is called a
continuous field if the function $t \mapsto A_t$ is norm continuous
on $T$. We assume that $\mu(t)$ is a Radon measure on $T$ with
$\mu(T)=1$. If the function $t \mapsto \|A_t\|$ is integrable, one
can form the Bochner integral $\int_{T}A_t{\rm d}\mu(t)$, which is
the unique element in ${\mathscr A}$ such that
$$\varphi\left(\int_TA_t{\rm d}\mu(t)\right)=\int_T\varphi(A_t){\rm d}\mu(t)$$
for every linear functional $\varphi$ in the norm dual ${\mathscr
A}^*$ of ${\mathscr A}$. It is easy to see that the set
$C(T,\mathscr{A})$ of all continuous fields of operators on $T$ with
values in ${\mathscr A}$ is a $C^*$-algebra under the pointwise
operations and the norm $\|(A_t)\|=\sup_{t\in T}\|A_t\|$; cf.
\cite{HPP}. Clearly $\mathscr{A}$ can be regarded as a
$C^*$-subalgebra of $C(T,\mathscr{A})$ via the constant fields. Then
the mapping $\Phi: C(T,\mathscr{A}) \to \mathscr{A}$ defined by
$\Phi\big((A_t)\big)=\int_TA_td\mu(t)$ satisfies the following
conditions:
\begin{itemize}
\item[(i)] $\Phi(X)=X$ for all $X \in \mathscr{A}$;
\item[(ii)] $\Phi(X\,(A_t)\, Y)=X\Phi((A_t))Y$ for all $X, Y \in
\mathscr{A}$ and all $(A_t)\in C(T,\mathscr{A})$;
\item[(iii)] If $(A_t) \geq 0$, then $\Phi((A_t)) \geq 0$.
\end{itemize}
Thus it is a conditional expectation and so it is completely
positive. Applying Theorem \ref{main2} we reach to the following
result.

\begin{corollary} Let $M_1, m_1, M_2, m_2 \in {\Bbb
C}$ and fields $(A_t)$ and $(B_t)$ of $C(T,\mathscr{A})$ belong to
the balls of diameter $[m_1I, M_1I]$ and $[m_2I, M_2I]$,
respectively, where $I$ denotes the identity element of $
C(T,\mathscr{A})$. Then
\begin{eqnarray*}
\left| \int_TA_tB_td\mu(t)-\int_TA_td\mu(t)\int_TB_td\mu(t)\right|
\leq \frac{1}{4} |M_1-m_1| \, |M_2-m_2|I.
\end{eqnarray*}
\end{corollary}
In the discrete case $T = \{1, \cdots, n\}$ we get
\begin{corollary} Let self-adjoint elements $A_1, \cdots, A_n, B_1, \cdots, B_n \in \mathscr{A}$ satisfy
\[
m_1\leq A_j \leq M_1,\qquad m_2\leq B_j \leq M_2\qquad(j=1, \cdots,
n)
\] for some real numbers $m_1, m_2, M_1, M_2$. If $C_1, \cdots, C_n \in \mathscr{A}$ are such that $\sum_{j=1}^nC_j^*C_j=I$, then
\begin{eqnarray*}
\left|\sum_{j=1}^nC_j^*A_jB_jC_j-\sum_{j=1}^nC_j^*A_jC_j\sum_{j=1}^nC_j^*B_jC_j\right|
\leq \frac{1}{4} (M_1-m_1) (M_2-m_2)I.
\end{eqnarray*}
\end{corollary}
The last inequality is clearly a generalization of the discrete case
of the integral version \eqref{MSM} of the Gr\"uss inequality which
asserts that if $m_1\leq a_j \leq M_1, m_2\leq b_j \leq M_2\,\,(j=1,
\cdots, n)$ are real numbers, then
\begin{eqnarray*}
\left|\frac{1}{n}
\sum_{j=1}^na_jb_j-\frac{1}{n}\sum_{j=1}^na_j\frac{1}{n}\sum_{j=1}^nb_j\right|
\leq \frac{1}{4} (M_1-m_1)(M_2-m_2).
\end{eqnarray*}

It's worth mentioning here that a more precise estimate of the discrete Gr\"uss inequality is the inequality by
Biernacki, Pidek and Ryll-Nardjewski \cite{BPR} which states that if $m_1\leq a_j \leq M_1, m_2\leq b_j \leq M_2\,\,(j=1,
\cdots, n)$ are real numbers, then

\begin{eqnarray*}
\left|\frac{1}{n}
\sum_{j=1}^na_jb_j-\frac{1}{n}\sum_{j=1}^na_j\frac{1}{n}\sum_{j=1}^nb_j\right|
\leq \frac{1}{n} \left[ \frac{n}{2}\right] \left( 1- \frac{1}{n}\left[ \frac{n}{2}\right]\right) (M_1-m_1)(M_2-m_2).
\end{eqnarray*}


\begin{thebibliography}{99}

\bibitem{ACM} \textsc{M. Amyari, M. Chakoshi and M.S. Moslehian}, Quasi-representations of Finsler modules over $C^*$-algebras, \textit{J. Operator Theory}, \textbf{70} (2013), no. 1, 181--190.

\bibitem{AND} \textsc{T. Ando}, \textit{Topics on Operator Inequalities}, Division of Applied Mathematics, Research Institute of Applied Electricity, Hokkaido University, Sapporo, 1978.

\bibitem{B-I-V} \textsc{S. Bani\'c, D. Ili\v{s}evi\'c and S. Varo\v{s}anec}, {Bessel- and Gr\"uss-type inequalities in
inner product modules}, \textit{Proc. Edinb. Math. Soc.}, \textbf{(2) 50}
(2007), no. 1, 23--36.

\bibitem{bhatiaa} \textsc{R. Bhatia}, \textit{Matrix Analysis}, Springer Verlag, New York, 1997.

\bibitem{bhatiab} \textsc{R. Bhatia}, \textit{Positive Definite Matrices}, Princeton Series in Applied Mathematics. Princeton University Press, Princeton, NJ, 2007.

\bibitem{BPR} \textsc{M. Biernacki, H. Pidek and C. Ryll-Nardzewski}, \textit{Ann. Univ. Mariae Curie-Sklodowska}, \textbf{A4}(1950), 1--4.

\bibitem{BOU} \textsc{J.-C. Bourin}, {Matrix versions of some classical inequalities}, \textit{Linear Algebra Appl.}, \textbf{416} (2006), no. 2-3, 890--907.

\bibitem{CHO} \textsc{M.D. Choi}, {A Schwarz inequality for poitive linear maps on $C^*$- algebras}, \textit{Illinois J. Math.} \textbf{18}(1974) 565--574.

\bibitem{DRA2} \textsc{S.S. Dragomir}, \textit{Operator inequalities of the Jensen, \v{C}eby\v{s}ev and Gr\"uss
type}, Springer Briefs in Mathematics, Springer, New York, 2012.

\bibitem{GK} I.C. Gohberg and M.G. Krein, \textit{Introduction to the theory of
linear nonselfadjoint operators}, Transl. Math. Monogr, 18,
Providence, R.I: Amer. Math. Soc., 1969.

\bibitem{GRU} \textsc{G. Gr\"uss}, {\"Uber das Maximum des absoluten Betrages von
$\frac{1}{b-a}\int_{a}^{b}f(x)g(x)dx-\frac{1}{(b-a)^{2}}\int_{a}^{b}f(x)dx\int_{a}^{b}g(x)dx$},
\textit{Math. Z.}, \textbf{39} (1935), 215--226.

\bibitem{HPP} \textsc{F. Hansen, I. Peri\'c and J. Pe\v{c}ari\'c}, {Jensen's operator
inequality and its converses}, \textit{Math. Scand.} \textbf{100} (2007), no.
1, 61--73.

\bibitem{JKM} \textsc{D. Jocic, D. Krtinic and M.S. Moslehian}, {Landau and Gr\"uss type
inequalities for inner product type integral transformers in norm
ideals}, \textit{Math. Inequal. Appl.} \textbf{16} (2013), no. 1, 109--125.

\bibitem{K-P} \textsc{R.V. Kadison and G.K. Pedersen}, {Means and convex combinations of unitary operators},
\textit{Math. Scan.}, \textbf{57} (1985), 249--266.

\bibitem{P-R} \textsc{I. Peri\'c and R. Raji\'c}, {Gr\"uss inequality
for completely bounded maps}, \textit{Linear Algebra Appl.}, \textbf{390}
(2004), 287--292.

\bibitem{M-R}\textsc{M.S. Moslehian and R. Raji\'c}, {A Gr\"uss inequality for
$n$-positive linear maps}, \textit{Linear Algebra Appl.}, \textbf{433} (2010),
1555--1560.

\bibitem{MNS} \textsc{M.S. Moslehian, R. Nakamoto and Y. Seo}, {A Diaz--Metcalf type
inequality for positive linear maps and its applications}, \textit{Electron.
J. Linear Algebra}, \textbf{22} (2011), 179--190.

\bibitem{REN} \textsc{P.F. Renaud}, {A matrix formulation of Gr\"uss inequality}, \textit{Linear
Algebra Appl.}, \textbf{335} (2001) 95--100.

\bibitem{STA} \textsc{J.G. Stampfli}, {The norm of a derivation}, \textit{Pacific J. Math.},
\textbf{33} (1970), 737 --747.

\bibitem {STI} \textsc{W.F. Stinespring}, {Positive functions on
$C^*$-algebras}, \textit{Proc. Amer. Math. Soc.}, \textbf{6} (1955), 211--216.

\bibitem{ZHA} \textsc{F. Zhang}, \textit{Matrix Theory, Basic results and techniques}, Second edition. Universitext. Springer, New York, 2011.

\end{thebibliography}
\end{document}